\documentclass[11pt]{amsart}

\usepackage{amssymb,amsmath,amsthm, color}
\usepackage{geometry}
\usepackage[normalem]{ulem}
\usepackage[makeroom]{cancel}

\title[order statistics of  log-concave vectors]{ Two-sided estimates for order statistics
\\ of  log-concave random  vectors}
\author{Rafa{\l} Lata{\l}a and Marta Strzelecka}
\date{January 6, 2019}
\address{Institute of Mathematics, University of Warsaw, Banacha 2, 02--097 Warsaw, Poland.}
\email{rlatala@mimuw.edu.pl, martast@mimuw.edu.pl}
\thanks{The research of RL was supported by the National Science Centre, Poland grant 2015/18/A/ST1/00553 and
of MS by  the National Science Centre, Poland grant 2015/19/N/ST1/02661}
\newtheorem{thm}{Theorem}
\newtheorem{prop}[thm]{Proposition}
\newtheorem{cor}[thm]{Corollary}
\newtheorem{lem}[thm]{Lemma}
\newtheorem{rmk}[thm]{Remark}

\def\Ex{{\mathbb E}}
\def\Pr{{\mathbb P}}
\def\er{{\mathbb R}}
\def\ind{{\mathbf 1}}
\def\eps{\varepsilon} 
\def\ve{\varepsilon}

\def\Cov{{\operatorname {Cov}}}
\def\Id{{\operatorname {Id}}}

\begin{document}

\begin{abstract}
We establish two-sided bounds for expectations of order statistics ($k$-th maxima) of moduli of coordinates of 
centered log-concave random vectors with uncorrelated coordinates. Our bounds are exact up to  multiplicative universal constants in the unconditional case for all $k$ and in the isotropic case for $k \leq n-cn^{5/6}$. We also derive two-sided estimates for expectations of sums of $k$ largest moduli of coordinates for some classes of random vectors.
\end{abstract}

\maketitle

\section{Introduction and main results}

For a vector $x\in \er^n$  let $k{\text -}\max x_i$ (or $k{\text -}\min x_i$) denote its \emph{$k$-th maximum} (respectively its \emph{ $k$-th minimum}), i.e.  its $k$-th maximal (respectively $k$-th minimal) coordinate. For a random vector $X=(X_1,\ldots,X_n)$, $k{\text -}\min X_i$ is also called the $k$-th order statistic of $X$.

Let $X=(X_1,\ldots,X_n)$ be a random vector with finite first moment. In this note we try to estimate 
$\Ex k{\text -}\max_i|X_i|$ and 
\[
\Ex \max_{|I|=k}\sum_{i\in I}|X_i|=\Ex \sum_{l=1}^k l\text{-}\max_i|X_i|.
\]
Order statistics play an important role in various statistical applications and there is an extensive literature
on this subject (cf.\ \cite{BC,DN} and references therein).

We put special emphasis  on the case of log-concave vectors, i.e.  random vectors $X$ 
satisfying the property 
$\Pr(X\in \lambda K+(1-\lambda)L)\geq \Pr(X\in K)^\lambda\Pr(X\in L)^{1-\lambda}$ for any $\lambda \in [0,1]$ and any nonempty compact sets $K$ and $L$. By the result of Borell \cite{Bo} a vector $X$ with full dimensional support is 
log-concave if and only if it has a log-concave density, i.e. the density of a form $e^{-h(x)}$ where 
$h$ is convex with values in $(-\infty,\infty]$.   
A typical example of a log-concave vector is a vector uniformly distributed over a convex body.
In recent years the study of log-concave vectors attracted attention of many researchers,
cf. monographs  \cite{AGM,BGVV}.

To bound the sum of $k$ largest coordinates of $X$ we define
\begin{equation}
\label{eq:deftk}
t(k,X):=\inf\left\{t>0\colon \frac{1}{t}\sum_{i=1}^n\Ex |X_i|\ind_{\{|X_i|\geq t\}}\leq k\right\}.
\end{equation}
and start with an easy upper bound.

\begin{prop}\label{lem_easy}
For any random vector $X$ with finite first moment we have
\begin{equation}
\label{eq:easy}
\Ex \max_{|I|=k}\sum_{i\in I}|X_i|\leq 2kt(k,X).
\end{equation}
\end{prop}

\begin{proof}
For any $t>0$ we have
\[
\max_{|I|=k}\sum_{i\in I}|X_i|\leq tk+\sum_{i=1}^n|X_i|\ind_{\{|X_i|\geq t\}}. \qedhere
\]
\end{proof}

It turns out that this bound may be reversed for vectors with independent coordinates or, more generally,
vectors satisfying the following condition
\begin{equation}
\label{cond_neg_cor}
\Pr (|X_i|\ge s, |X_j| \ge t) \le \alpha \Pr(|X_i|\ge s)\Pr(|X_j|\ge t) 
\qquad \mbox{for all } i\neq j \mbox{ and all } s,t>0.
\end{equation}	
If $\alpha=1$ this means that moduli of coordinates of $X$ are negatively correlated.

\begin{thm}
\label{thm:neg_cor}
Suppose that a random vector $X$ satisfies condition \eqref{cond_neg_cor} with some $\alpha\geq 1$. Then there exists a
constant $c(\alpha)>0$ which depends only on $\alpha$ such that
for any $1\leq k\leq n$,
\[
c(\alpha)kt(k,X)\leq \Ex\max_{|I|=k} \sum_{i\in I} |X_i|\leq 2kt(k,X).
\]
We may take $c(\alpha)=( 36(5+4\alpha)(1+2\alpha))^{-1}$.
\end{thm}

In the case of i.i.d. coordinates  two-sided bounds for $\Ex\max_{|I|=k} \sum_{i\in I} | a_iX_i|$ in terms of an Orlicz norm (related to the distribution of $X_i$) of a vector $( a_i)_{i\le n}$ where known before, see \cite{GLSW1}. 

Log-concave vectors with diagonal covariance matrices behave in many aspects like vectors with independent
coordinates. This is true also in our case.

\begin{thm}
\label{thm:kmaxlogconc}
Let $X$ be a log-concave random vector with uncorrelated coordinates 
(i.e. $\Cov(X_{i},X_{j})=0$ for $i\neq j$). Then for any  $1\leq k\leq n$,
\[
ckt(k,X)\leq \Ex\max_{|I|=k}\sum_{i\in I}|X_i|\leq 2kt(k,X).
\]
\end{thm}

 In the above statement and in the sequel $c$ and $C$ denote positive universal constants.

 The next two examples show that the lower bound cannot hold if $n \gg k$ and only marginal distributions of
$X_i$ are log-concave or the coordinates of $X$  are highly correlated.

\medskip

\noindent
{\bf Example 1.} Let $X=(\ve_1g,\ve_2g,\ldots,\ve_ng)$, where $\ve_1,\ldots,\ve_n,g$ are independent,
$\Pr(\ve_i=\pm 1)=1/2$ and $g$ has the normal ${\mathcal N}(0,1)$ distribution. Then $\Cov X = \Id$ and it is not hard to check that  $\Ex\max_{|I|=k}\sum_{i\in I}|X_i|=k \sqrt{2/\pi}$ and $t(k, X)\sim \ln^{1/2} (n/k)$ if $k\leq n/2$. 

\medskip

\noindent
{\bf Example 2.} Let $X=(g,\ldots,g)$, where $g\sim {\mathcal N}(0,1)$.  Then, as in the previous example, 
$\Ex\max_{|I|=k}\sum_{i\in I}|X_i|=k\sqrt{2/\pi}$ and $t(k, X)\sim \ln^{1/2} (n/k)$.

\medskip

\noindent
{\bf Question 1.} Let $X'=(X'_1,X'_2,\ldots,X'_n)$ be a decoupled version of $X$, 
i.e. $X'_i$ are independent and  $X'_i$ has the same distribution as $X_i$. 
Due to Theorem \ref{thm:neg_cor} (applied to $X'$), the assertion of
 Theorem \ref{thm:kmaxlogconc} may be stated equivalently as
\[
\Ex \max_{|I|=k}\sum_{i\in I}|X_i|\sim
\Ex \max_{|I|=k}\sum_{i\in I}|X'_i|.
\]
Is the more general fact true that for any symmetric norm and any log-concave vector $X$ with uncorrelated
coordinates
\[
\Ex\|X\|\sim \Ex\|X'\|?
\]
Maybe such an estimate holds at least in the case of unconditional log-concave vectors? 

\medskip

We turn our attention to bounding $k$-maxima of $|X_i|$.  This was investigated   in \cite{GLSW2}  (under some strong assumptions  on the function $t\mapsto \Pr(|X_i|\ge t)$) and  in 
the  weighted i.i.d. setting  in \cite{GLSW1, GLSW3, PR}. We will give different bounds valid for log-concave vectors, in which we do not have to assume independence, nor any special conditions on the growth of the distribution function of the coordinates of $X$. To this end we need to define
another quantity:
\[
t^*(p,X):=\inf\biggl\{t>0\colon\ \sum_{i=1}^n\Pr(|X_i|\geq t)\leq p\biggr\}\quad \mbox{ for }0<p<n.
\]

\begin{thm}
\label{thm:estkmax}
Let $X$ be  a mean zero log-concave $n$-dimensional random vector with uncorrelated coordinates and $1\leq k\leq n$. Then 
\[
\Ex k\text{-}\max_{i\leq n}|X_i|\geq 
\frac{1}{2}\mathrm{Med}\Bigl( k\text{-}\max_{i\leq n}|X_i|\Bigr) \geq ct^*\biggl(k-\frac12,X\biggr).
\]
Moreover, if $X$ is additionally unconditional then
\[
\Ex k\text{-}\max_{i\leq n}|X_i|\leq Ct^*\biggl(k-\frac12,X\biggr).
\]
\end{thm}

The next theorem provides an upper bound in the general log-concave case. 

\begin{thm}
\label{thm:revestkmax}
Let $X$ be  a mean zero log-concave $n$-dimensional random vector with uncorrelated coordinates and $1\leq k\leq n$. Then 
\begin{equation}
\label{eq:kmaxupgen1}
\Pr\biggl(k\text{-}\max_{i\leq n}|X_i|\geq Ct^*\biggl(k-\frac12, X\biggr)\biggr)\leq 1-c
\end{equation}
and
\begin{equation}
\label{eq:kmaxupgen2}
\Ex k\text{-}\max_{i\leq n}|X_i|\leq Ct^*\biggl(k-\frac{1}{2}k^{5/6}, X\biggr).
\end{equation}
\end{thm}

In the isotropic case (i.e. $\Ex X_i=0, \Cov X = \Id$) one may show that 
$t^*( k/2,X)\sim t^*(k,X)\sim t(k,X)$ for $k\leq n/2$ and 
$t^*( p,X)\sim \frac{n- p}{n}$ for $ p\geq n/4$ (see Lemma \ref{lem:asympt_tk} below). 
In particular $t^*(n-k+1-(n-k+1)^{5/6}/2, X)\sim k/n - n^{-1/6}$ for $k\leq n/2$. This together with the two previous theorems implies the following corollary.

\begin{cor}
\label{cor:estkmaxisotr}
Let $X$ be  an isotropic log-concave $n$-dimensional random vector and $1\leq k\leq n/2$. Then 
\[
\Ex k\text{-}\max_{i\leq n}|X_i|\sim t^{*}(k,X)\sim t(k,X)
\]
and
\[
c\frac{k}{n}\leq \Ex k\text{-}\min_{i\leq n}|X_i|=\Ex (n-k+1)\text{-}\max_{i\leq n}|X_i|
\leq C\left(\frac{k}{n}+n^{-1/6}\right). 
\]
If $X$ is additionally unconditional then
\[
\Ex k\text{-}\min_{i\leq n}|X_i|=\Ex (n-k+1)\text{-}\max_{i\leq n}|X_i|\sim \frac{k}{n}.
\]
\end{cor}

\textbf{Question 2.} Does the second part of Theorem \ref{thm:estkmax} hold without the unconditionality assumptions? In particular, is it true that 
$\Ex k\text{-}\min_{i\leq n}|X_i|\sim k/n$ for $1\leq k\leq n/2$?

\medskip

\textbf{Notation.} Throughout this paper by letters $C, c$ we denote universal positive constants and by $C(\alpha), c(\alpha)$  constants depending
only on the parameter $\alpha$. The values of  constants $C,c,C(\alpha), c(\alpha)$ may differ at each occurrence. 
If we need to fix a value of constant, we use letters $C_0, C_1, \ldots$ or $c_0, c_1, \ldots$. 
We write $f\sim g$ if $cf\leq g\leq Cg$.
For a random variable $Z$ we denote $\|Z\|_p=(\Ex|Z|^p)^{1/p}$. Recall that a random vector $X$ is called isotropic, if $\Ex X=0$ and $\Cov X=\Id$.

This note is organised as follows. In Section \ref{sect:concentration} we provide a lower bound for the sum of $k$ largest coordinates, which involves the Poincar\'e constant of a vector. In Section \ref{sect:log-concave} we use this result to obtain Theorem \ref{thm:kmaxlogconc}. In Section \ref{sect:uncorrelated} we prove Theorem \ref{thm:neg_cor} and provide its application to comparison of weak and strong moments. In Section \ref{sect:order_statistics} we prove  the first part of  Theorem \ref{thm:estkmax}  and in Section \ref{sect:order_statistics_upper} we prove the second part of Theorem \ref{thm:estkmax}, Theorem \ref{thm:revestkmax}, and Lemma \ref{lem:asympt_tk}.

\section{Exponential concentration} \label{sect:concentration}

A probability measure $\mu$ on 
$\er^n$ satisfies \emph{exponential concentration with constant $\alpha>0$} if for any Borel set
$A$ with $\mu(A)\geq 1/2$, 
\[
1-\mu(A+uB_2^n)\leq e^{-u/\alpha}\quad \mbox{ for all }u>0. 
\] 
We say that a random $n$-dimensional vector satisfies exponential concentration if its distribution
has such a property.

It is well known that exponential concentration is implied by the Poincar\'e inequality
\[
\mathrm{Var}_\mu f \leq \beta  \int |\nabla f|^2d\mu \quad \mbox{ for all bounded smooth  functions }
f\colon\er^n\mapsto \er
\]
and $\alpha\leq 3\sqrt{\beta}$ (cf. \cite[Corollary 3.2]{Le}).

Obviously, the constant in the exponential concentration is not linearly invariant. Typically one
assumes that the vector is isotropic. For our purposes a
more natural normalization will be that all coordinates have $L_1$-norm equal to  $1$.

The next proposition states that bound \eqref{eq:easy} may be reversed under the assumption that $X$ satisfies the exponential concentration.

\begin{prop}
\label{prop:maxkexpconc}
Assume that $Y=(Y_1,\ldots,Y_n)$ satisfies the exponential concentration with constant $\alpha > 0$ and 
$\Ex |Y_i|\geq 1$ for all $i$. Then for any sequence $a=(a_i)_{i=1}^n$ of real numbers 
and $X_i:=a_iY_i$ we have
\[
\Ex \max_{|I|=k}\sum_{i\in I}|X_i|\geq  \Bigl(8+64\frac{\alpha}{\sqrt{k}}\Bigr)^{-1}kt(k,X),
\]
where $t(k,X)$ is given by \eqref{eq:deftk}.
\end{prop}

We begin the proof with a few simple observations.

\begin{lem}
\label{lem:easy1}
For any real numbers $z_1,\ldots,z_n$ and $1\leq k\leq n$ we have
\[
\max_{|I|=k}\sum_{i\in I}|z_i|=\int_0^\infty \min\biggl\{k,\sum_{i=1}^n\ind_{\{|z_i|\geq s\}}\biggr\}ds.
\]
\end{lem}

\begin{proof}
Without loss of generality we may assume that $z_1\geq z_2\geq\ldots\geq z_n\geq 0$. Then
\begin{align*}
\int_0^\infty \min\biggl\{k,\sum_{i=1}^n\ind_{\{|z_i|\geq s\}}\biggr\}ds
&=\sum_{l=1}^{k-1}\int _{z_{l+1}}^{z_l}lds+\int_0^{z_k}kds
=\sum_{l=1}^{k-1}l(z_{l}-z_{l+1})+kz_k
\\
&=z_1+\ldots+z_k=\max_{|I|=k}\sum_{i\in I}|z_i|.  \qedhere
\end{align*}
\end{proof}

Fix a sequence $(X_i)_{i\leq n}$ and define for $s\geq 0$,
\begin{equation}
\label{eq:defN}
N(s):=\sum_{i=1}^n \ind_{\{|X_i|\geq s\}}.
\end{equation}

\begin{cor}
\label{byparts}
For any $k=1,\ldots,n$,
\[
\Ex\max_{|I|=k}\sum_{i\in I}|X_i|=\int_0^\infty\sum_{l=1}^k\Pr(N(s)\geq l)ds,
\]
and for any $t>0$,
\[
\Ex\sum_{i=1}^n|X_i|\ind_{\{|X_i|\geq t\}}=t\Ex N(t)+\int_{t}^\infty \sum_{l=1}^\infty\Pr(N(s)\geq l)ds.
\]
In particular
\[
\Ex\sum_{i=1}^n|X_i|\ind_{\{|X_i|\geq t\}}
\leq \Ex\max_{|I|=k}\sum_{i\in I}|X_i|+
\sum_{l=k+1}^\infty \left(t\Pr(N(t)\geq l)+\int_t^\infty\Pr(N(s)\geq l)ds\right).
\]
\end{cor}

\begin{proof}
We have
\begin{align*}
\int_0^\infty\sum_{l=1}^k\Pr(N(s)\geq l)ds&=\int_0^\infty\Ex\min\{k,N(s)\}ds
=\Ex\int_0^\infty\min\{k,N(s)\}ds
\\
&=\Ex\max_{|I|=k}\sum_{i\in I}|X_i|,
\end{align*}
where the last equality follows by Lemma \ref{lem:easy1}.

Moreover,
\begin{align*}
t\Ex N(t)+\int_{t}^\infty \sum_{l=1}^\infty\Pr(N(s)\geq l)ds
&=t\Ex N(t)+\int_{t}^\infty \Ex N(s)ds
\\
&=\Ex\sum_{i=1}^n \left(t\ind_{\{|X_i|\geq t\}}+\int_{t}^\infty\ind_{\{|X_i|\geq s\}}ds\right)
\\
&=\Ex\sum_{i=1}^n|X_i|\ind_{\{|X_i|\geq t\}}.
\end{align*}

The last part of the assertion easily follows, since 
\[
t\Ex N(t)=t\sum_{l=1}^n\Pr(N(t)\geq l)
\leq \int_0^t \sum_{l=1}^k\Pr(N(s)\geq l)ds+\sum_{l=k+1}^\infty t\Pr(N(t)\geq l). \qedhere
\]
\end{proof}

\begin{proof}[Proof of Proposition \ref{prop:maxkexpconc}]
To shorten the notation put $t_k:=t(k,X)$.
Without loss of generality we may assume that $a_1\geq a_2\geq\ldots\geq a_n\geq 0$ and $a_{\lceil k/4\rceil}=1$.
Observe first that
\[
\Ex \max_{|I|=k}\sum_{i\in I}|X_i|\geq \sum_{i=1}^{\lceil k/4\rceil}a_i\Ex |Y_i|\geq k/4,
\]
so we may assume that $t_k\geq 16\alpha/\sqrt{k}$.

Let $\mu$ be the law of $Y$ and
\[
A:=\Biggl\{y\in \er^n\colon\ \sum_{i=1}^n\ind_{\{|a_iy_i|\geq \frac{1}{2}t_k\}}< \frac{k}{2}\Biggr\}.
\]
We have
\[
\Ex\max_{|I|=k}\sum_{i\in I}|X_i|
\geq \frac{k}{4}t_k\Pr\Biggl(\sum_{i=1}^k\ind_{\{|a_iY_i|\geq \frac{1}{2}t_k\}}\geq \frac{k}{2}\Biggr)
= \frac{k}{4}t_k(1-\mu(A)),
\]
so we may assume that $\mu(A)\geq 1/2$.

Observe that if $y\in A$ and $\sum_{i=1}^n \ind_{\{|a_iz_i|\geq s\}}\geq l >k$ for some $s\geq t_k$ then
\[
\sum_{i=1}^n(z_i-y_i)^2\geq \sum_{i=\lceil k/4\rceil}^n(a_iz_i-a_iy_i)^2\geq 
(l-3k/4)(s-t_k/2)^2>\frac{ls^2}{16}.
\]
Thus we have
\[
\Pr(N(s)\geq l)\leq 1-\mu\biggl(A+\frac{s\sqrt{l}}{4}B_2^n\biggr)\leq e^{-\frac{s\sqrt{l}}{4\alpha}}
\quad \mbox{ for }l > k,\ s\geq t_k.
\]
Therefore  
\[
\int_{t_k}^\infty\Pr(N(s)\geq l)ds
\leq \int_{t_k}^\infty e^{-\frac{s\sqrt{l}}{4\alpha}}ds
=\frac{4\alpha}{\sqrt{l}}e^{-\frac{t_k\sqrt{l}}{4\alpha}}\quad  \mbox{for }l > k,
\]
and
\begin{align*}
\sum_{l=k+1}^\infty &\biggl(t_k\Pr(N(t_k)\geq l)+\int_{t_k}^\infty\Pr(N(s)\geq l)ds\biggr)
\leq \sum_{l=k+1}^\infty\biggl(t_k+\frac{4\alpha}{\sqrt{l}}\biggr) 
e^{-\frac{t_k\sqrt{l}}{4\alpha}}
\\
&\leq \biggl(t_k+\frac{4\alpha}{\sqrt{k+1}}\biggr)\int_k^\infty e^{-\frac{t_k\sqrt{u}}{4\alpha}}du
\leq \biggl(t_k+\frac{4\alpha}{\sqrt{k+1}}\biggr)e^{-\frac{t_k\sqrt{k}}{4\sqrt2 \alpha}}
\int_k^\infty e^{-\frac{t_k\sqrt{u-k}}{4\sqrt2 \alpha}}du
\\
&=\biggl(t_k+\frac{4\alpha}{\sqrt{k+1}}\biggr)\frac{64\alpha^2}{t_k^2}e^{-\frac{t_k\sqrt{k}}{4\sqrt2 \alpha}}
\leq \Bigl(t_k+\frac{1}{4}t_k\Bigr)\frac{k}{4}\leq  \frac{1}{2}kt_k,
\end{align*}
where to get the next-to-last inequality we used the fact that $t_k\geq 16\alpha/\sqrt{k}$.

Hence Corollary \ref{byparts} and the definition of $t_k$ yields
\begin{align*}
kt_k&\leq \Ex\sum_{i=1}^n|X_i|\ind_{\{|X_i|\geq t_k\}}
\\
&\leq \Ex\max_{|I|=k}\sum_{i\in I}|X_i|+
\sum_{l=k+1}^\infty \biggl(t_k\Pr(N(t_k)\geq l)+\int_{t_k}^\infty\Pr(N(s)\geq l)ds\biggr)
\\
&\leq \Ex\max_{|I|=k}\sum_{i\in I}|X_i|+\frac{1}{2}kt_k,
\end{align*}
so $\Ex\max_{|I|=k}\sum_{i\in I}|X_i|\geq \frac{1}{2}kt_k$.
\end{proof}

We finish this section with a simple fact that will be used in the sequel.

\begin{lem}
\label{lem:koncsmallmeas}
Suppose that a measure $\mu$ satisfies exponential concentration with constant $\alpha$. Then
for any $c\in (0,1)$ and any Borel set $A$ with $\mu(A)> c$ we have
\[
1-\mu(A+uB_2^n)\leq \exp\biggl(-\Bigl(\frac{u}{\alpha}+\ln c\Bigr)_+\biggr)\quad
\mbox{ for }u\geq 0.
\]
\end{lem}

\begin{proof}
Let $D:=\er^n\setminus(A+rB_2^n)$. Observe that $D+rB_2^n$ has an empty intersection with $A$ so if $\mu(D)\geq 1/2$ then
\[
c<\mu(A)\leq 1-\mu(D+rB_2^n)\leq e^{-r/\alpha},
\]
and $r<\alpha\ln(1/c)$. Hence $\mu(A+\alpha\ln(1/c)B_2^n)\geq 1/2$, therefore for $s\geq 0$,
\[
1-\mu(A+(s+\alpha\ln(1/c))B_2^n)= 1-\mu((A+\alpha\ln(1/c)B_2^n)+sB_2^n)\leq e^{-s/\alpha},
\]
and the assertion easily follows.
\end{proof}

\section{Sums of largest coordinates of log-concave vectors} \label{sect:log-concave}

We will usethe regular growth of  moments of norms of log-concave vectors multiple times. By \cite[Theorem 2.4.6]{BGVV}, if $f:\er^n \to \er$ is a seminorm, and $X$ is log-concave, then 
\begin{equation}
\label{eq:reg_mom}
(\Ex f( X)^p)^{1/p}\leq C_1\frac{p}{q}(\Ex f( X)^q)^{1/q} \quad \text{for } p\geq q\geq 2,
\end{equation}
where $C_1$ is a universal constant. 

We will also apply a few times the functional version of the  Gr\"unbaum inequality (see \cite[Lemma 5.4]{LoV})
which states that 
\begin{equation}
\label{Grunbaum}
\Pr(Z\geq 0)\geq \frac{1}{e} \quad \mbox{ for any mean-zero log-concave random variable Z.}
\end{equation}

Let us start with a few technical lemmas. The first one will be used to reduce the proof of 
Theorem \ref{thm:kmaxlogconc} to the symmetric case.

\begin{lem}
\label{lem:symm}
Let $X$ be a log-concave $n$-dimensional vector and $X'$ be an independent copy of $X$. Then 
for any $1\leq k\leq n$,
\[
\Ex\max_{|I|=k}\sum_{i\in I}|X_i-X_i'|\leq 2\Ex\max_{|I|=k}\sum_{i\in I}|X_i|
\]
and 
\begin{equation}
\label{eq:symmtk}
t(k,X)\leq et(k,X-X')+\frac{2}{k}\max_{|I|=k}\sum_{i\in I}\Ex |X_i|.
\end{equation}
\end{lem}

\begin{proof}
The first estimate follows by the easy bound
\[
\Ex\max_{|I|=k}\sum_{i\in I}|X_i-X_i'|\leq \Ex\max_{|I|=k}\sum_{i\in I}|X_i|+
\Ex\max_{|I|=k}\sum_{i\in I}|X_i'|=2\Ex\max_{|I|=k}\sum_{i\in I}|X_i|.
\]

To get the second bound we may and will assume that $\Ex |X_1|\geq \Ex |X_2|\geq \ldots\geq \Ex |X_n|$.
Let us define $Y:=X-\Ex X$, $Y':=X'-\Ex X$ and 
$M:=\frac{1}{k}\sum_{i=1}^k\Ex|X_i|\geq \max_{i\geq k}\Ex|X_i|$.
Obviously
\begin{equation}
\label{eq:firstk}
\sum_{i=1}^k\Ex |X_i|\ind_{\{|X_i|\geq t\}}\leq kM \quad\mbox{for }t\geq 0.
\end{equation}
 
We have $\Ex Y_i=0$, thus $\Pr(Y_i\leq 0)\geq 1/e$ by \eqref{Grunbaum}. Hence
\[
\Ex Y_i\ind_{\{Y_i > t\}}\leq e\Ex Y_i\ind_{\{Y_i > t,Y_i'\leq 0\}}
\leq e\Ex |Y_i-Y_i'|\ind_{\{Y_i-Y_i' > t\}}=e\Ex |X_i-X_i'|\ind_{\{X_i-X_i' > t\}}
\]
for $t\ge 0$.
In the same way we show that 
\[
\Ex |Y_i|\ind_{\{Y_i < -t\}}\leq e\Ex |Y_i|\ind_{\{Y_i <-t,Y_i'\geq 0\}}
\leq e\Ex |X_i-X_i'|\ind_{\{X_i'-X_i > t\}}
\]
Therefore 
\[
\Ex |Y_i|\ind_{\{|Y_i| > t\}}\leq e\Ex |X_i-X_i'|\ind_{\{|X_i-X_i'| > t\}}.
\]

 We have
\begin{align*}
\sum_{i=k+1}^n&\Ex |X_i|\ind_{\{|X_i| > e t(k, X-X')+M\}}
\leq 
\sum_{i=k+1}^n\Ex |X_i|\ind_{\{|Y_i| > et(k,X-X')\}}
\\
&\leq \sum_{i=k+1}^n\Ex |Y_i|\ind_{\{|Y_i| > t(k,X-X')\}}
+\sum_{i=k+1}^n|\Ex X_i|\Pr(|Y_i| > et(k,X-X'))
\\
&\leq e\sum_{i=1}^n\Ex |X_i-X_i'|\ind_{\{|X_i-X_i'| > t(k,X-X')\}}
+M\sum_{i=1}^n\Pr(|Y_i| > et(k,X-X'))
\\
&\leq ekt(k,X-X')
+M\sum_{i=1}^n\ \bigl(et(k,X-X')\bigr)^{-1} \Ex |Y_i| \ind_{\{|Y_i| > et(k,X-X')\}}
\\
&\leq ekt(k,X-X')+M\sum_{i=1}^n\ t(k,X-X')^{-1} \Ex |X_i-X_i'| \ind_{\{|X_i-X_i'|  > t(k,X-X')\}}
\\
&\leq ekt(k,X-X')+kM.
\end{align*}
Together with \eqref{eq:firstk} we get
\[
\sum_{i=1}^n\Ex |X_i|\ind_{\{|X_i| > et(k,X-X')+M\}}\leq k(et(k,X-X')+2M)
\]
and \eqref{eq:symmtk} easily follows.
\end{proof}

\begin{lem}
\label{lem:cutmean}
Suppose that $V$ is a real symmetric log-concave random variable. Then for any $t>0$ and $\lambda\in (0,1]$,
\[
\Ex |V|\ind_{\{|V|\geq t\}} \leq \frac{4}{\lambda}\Pr(|V|\geq t)^{1-\lambda} \Ex |V|\ind_{\{|V|\geq \lambda t\}}.
\]
Moreover, if $\Pr(|V|\ge t)\le 1/4$, then $\Ex |V|\ind_{\{|V|\geq t\}} \leq 4t \Pr(|V|\geq t).$
\end{lem}

\begin{proof}

Without loss of generality we may assume that $\Pr(|V|\geq t)\leq 1/4$ (otherwise the  first 
estimate is trivial).

Observe that $\Pr(|V|\geq s)=\exp(-N(s))$ where $N\colon [0,\infty)\to [0,\infty]$ is convex and
$N(0)=0$. In particular
\[
\Pr(|V|\geq \lambda t)\leq \Pr(|V|\geq t)^{\lambda}\quad \mbox{for }\lambda>1
\] 
and
\[
\Pr(|V|\geq \lambda t)\geq \Pr(|V|\geq t)^{\lambda}\quad \mbox{for }\lambda\in [0,1].
\] 

We have
\begin{align*}
\Ex |V|\ind_{\{|V|\geq t\}} 
&\leq \sum_{k=0}^\infty 2^{(k+1)}t\Pr(|V|\geq 2^k t)
\leq  2t \sum_{k=0}^\infty 2^{k}\Pr(|V|\geq t)^{2^k}
\\
&\leq 2t \Pr(|V|\geq t)\sum_{k=0}^\infty 2^{k}4^{1-2^k}\leq 4t \Pr(|V|\geq t).
\end{align*}
This implies the second part of the lemma.

 To conclude the proof of the first bound it is enough to observe that
\[
\Ex |V|\ind_{\{|V|\geq \lambda t\}}\geq \lambda t\Pr(|V|\geq \lambda t)\geq \lambda t\Pr(|V|\geq t)^{\lambda}. \qedhere
\]
\end{proof}

\begin{proof}[Proof of Theorem \ref{thm:kmaxlogconc}]
By Proposition \ref{lem_easy} it is enough to show the lower bound. By Lemma \ref{lem:symm} we may assume that $X$ is symmetric. 
We may also obviously assume that $\|X_i\|_2^2=\Ex X_i^2>0$ for all $i$.

Let $Z=(Z_1,\ldots,Z_n)$, where $Z_i=X_i/\|X_i\|_2$. Then $Z$ is log-concave, isotropic and, by  \eqref{eq:reg_mom},
$\Ex |Z_i|\geq 1/ (2C_1)$ for all $i$. Set $Y:=2C_1 Z$. Then $X_i=a_i Y_i$ and $\Ex |Y_i|\geq 1$.
Moreover, by the result of  Lee and Vempala \cite{LV}, we know that any $m$-dimensional projection of $Z$ is a log-concave, isotropic $m$-dimensional vector 
thus it satisfies the exponential concentration with a constants $Cm^{1/4}$.  (In fact an easy modification of the proof below shows that for our purposes it would be enough to have
exponential concentration with a constant  $Cm^{\gamma}$ for some $\gamma<1/2$, so one may also use Eldan's result
\cite{El} which gives such estimates for any $\gamma>1/3$).
So any $m$-dimensional projection of $Y$ satisfies exponential concentration with constant 
$C_2m^{1/4}$.

Let us fix $k$ and set $t:=t(k,X)$, then (since $X_i$ has no atoms)
\begin{equation}
\label{eq:t=tk}
\sum_{i=1}^n\Ex |X_i|\ind_{\{|X_i|\geq t\}}=kt.
\end{equation}

For $l=1,2,\ldots$ define
\[
I_l:=\{i\in [n]\colon\ \beta^{l-1}\geq \Pr(|X_i|\geq t)\geq \beta^l\},
\]
where $\beta=2^{-8}$.
By \eqref{eq:t=tk} there exists $l$ such that 
\[
\sum_{i\in I_l}\Ex |X_i|\ind_{\{|X_i|\geq t\}}\geq kt2^{-l}.
\] 

Let us consider three cases. 

\noindent
(i) $l=1$ and $|I_1|\leq k$. Then
\[
\Ex\max_{|I|=k}\sum_{i\in I}|X_i|\geq \sum_{i\in I_1}\Ex |X_i|\ind_{\{|X_i|\geq t\}}\geq \frac{1}{2}kt.
\]

\noindent
(ii) $l=1$ and $|I_1|> k$. Choose $J\subset I_1$ of cardinality $k$. Then
\[
\Ex\max_{|I|=k}\sum_{i\in I}|X_i|\geq \sum_{i\in J}\Ex|X_i|\geq
\sum_{i\in J}t\Pr(|X_i|\geq t)\geq \beta kt.
\]

\noindent
(iii) $l>1$. By Lemma \ref{lem:cutmean} (applied with $\lambda=1/8$) we have
\begin{equation}\label{proof_log-conc_pom1}
\sum_{i\in I_l}\Ex |X_i|\ind_{\{|X_i|\geq t/8\}}\geq 
\frac{1}{32}\beta^{-7(l-1)/8}\sum_{i\in I_l}\Ex |X_i|\ind_{\{|X_i|\geq t\}}
\geq \frac{1}{32}\beta^{-7(l-1)/8}2^{-l}kt.
\end{equation}
 Moreover for $i\in I_l$,  $\Pr(|X_i|\geq t)\leq \beta^{l-1}\leq 1/4$, so
the second part of Lemma \ref{lem:cutmean} yields
\[
4t|I_l|\beta^{l-1}\geq \sum_{i\in I_l}\Ex |X_i|\ind_{\{|X_i|\geq t\}}\geq kt2^{-l}
\]
and $|I_l|\geq \beta^{1-l}2^{-l-2}k=2^{7l-10}k\geq k$.

Set $k':=\beta^{-7l/8}2^{-l}k=2^{6l}k$. If $k'\geq |I_l|$ then, using \eqref{proof_log-conc_pom1}, we estimate 
\[
\Ex\max_{|I|=k}\sum_{i\in I}|X_i|\geq \frac{k}{|I_l|}\sum_{ i\in I_l}\Ex |X_i|
\geq  \beta^{7l/8}2^l\sum_{i\in I_l}\Ex |X_i|\ind_{\{|X_i|\geq t/8\}}\geq \frac{1}{32}\beta^{7/8}kt
=2^{-12}kt.
\]

Otherwise set $X'=(X_i)_{i\in I_l}$ and $Y'=(Y_i)_{i\in I_l}$. By \eqref{eq:t=tk} we have
\[
kt\geq \sum_{i\in I_l}\Ex |X_i|\ind_{\{|X_i|\geq t\}}\geq |I_l|t\beta^{l},
\]
so $|I_l|\leq k\beta^{-l}$ and $Y'$ satisfies exponential concentration with constant
$\alpha'=C_2k^{1/4}\beta^{-l/4}$.  
Estimate \eqref{proof_log-conc_pom1} yields
\[
\sum_{i\in I_l}\Ex |X_i|\ind_{\{|X_i|\geq 2^{-12}t\}}\geq
\sum_{i\in I_l}\Ex |X_i|\ind_{\{|X_i|\geq t/8\}}\geq  2^{-12}k't,
\]
so $t(k',X')\geq 2^{-12}t$. Moreover, by Proposition \ref{prop:maxkexpconc} we have (since $k'\le |I_l|$)
\[
\Ex\max_{I\subset I_l,|I|=k'}\sum_{i\in I}|X_i|\geq \frac{1}{8+64\alpha'/\sqrt{k'}}k't(k',X').
\]
To conclude observe that
\[
\frac{\alpha'}{\sqrt{k'}}= C_2 2^{- l}k^{- 1/4}\leq \frac{C_2}4
\]
and since $k'\geq k$,
\[
\Ex\max_{|I|=k}\sum_{i\in I}|X_i|\geq \frac{k}{k'}\Ex\max_{I\subset I_l,|I|=k'}\sum_{i\in I}|X_i|
\geq \frac{1}{ 8+16C_2}2^{-12}tk. \qedhere
\]
\end{proof}

\section{Vectors satisfying condition \eqref{cond_neg_cor}} \label{sect:uncorrelated}

\begin{proof}[Proof of Theorem \ref{thm:neg_cor}]
By Proposition \ref{lem_easy} we need to show only the lower bound. Assume first that variables 
$X_i$ have no atoms and 
$k\ge 4(1+\alpha)$.

Let $t_k=t(k,X)$. Then $\Ex \sum_{i=1}^n |X_i| \ind_{\{|X_i| \ge t_k\}} = kt_k $.
Note, that \eqref{cond_neg_cor} implies that for all $i\neq j$ we have 
\begin{equation} 
\label{neg_cor_exp}
\Ex |X_iX_j| \ind_{\{|X_i|\ge t_k, |X_j|\ge t_k\}} 
\le  \alpha \Ex |X_i| \ind_{\{|X_i|\ge t_k\} }\Ex  |X_j|\ind_{\{|X_j|\ge t_k\} }.
\end{equation}
We may assume that $\Ex \max_{|I|=k} \sum_{i\in I} |X_i| \le \frac 16 kt_k$, because otherwise the lower bound holds trivially. 

Let us define 
\[
Y:= \sum_{i=1}^n |X_i| \ind_{\{kt_k \ge |X_i| \ge t_k\}} \quad \mbox{ and }\quad A:=(\Ex Y^2)^{1/2}.
\]
Since
\[
\Ex \max_{|I|=k} \sum_{i\in I} |X_i|  \ge \Ex \biggl[\frac{1}{2}kt_k \ind_{\{Y\ge kt_k/2\}}\biggr] 
= \frac{1}{2}kt_k \Pr\biggl(Y\ge \frac{kt_k}{2}\biggr),
\]
it suffices to bound  below the probability that $Y\ge kt_k/2$ by a constant depending only on $\alpha$.
	
We have 
\begin{align*}
A^2  
& = \Ex Y^2  
\le \sum_{i=1}^n \Ex X_i^2\ind_{\{kt_k\ge |X_i|\ge t_k\}} 
+ \sum_{i\neq j }\Ex |X_iX_j| \ind_{\{|X_i|\ge t_k, |X_j|\ge t_k\}}
\\ 
&\mathop{\le}^{\eqref{neg_cor_exp}} 
kt_k\Ex Y +   \alpha \sum_{i\neq j } \Ex |X_i| \ind_{\{|X_i|\ge t_k \}}\Ex  |X_j|\ind_{\{|X_j|\ge t_k\}}
\\ 
& \le kt_kA + \alpha \biggl( \sum_{i=1}^n \Ex|X_i|\ind_{\{|X_i|\ge t_k\}} \biggr)^2 
\le  \frac12 (k^2t_k^2+A^2) + \alpha k^2t_k^2.
\end{align*}
Therefore $A^2 \le  (1+2\alpha)k^2t_k^2$ and for any $l\ge k/2$ we have
\begin{align}
\notag
\Ex Y\ind_{\{Y\ge kt_k/2\}} 
&  \le lt_k\Pr(Y\ge kt_k/2) + \frac{1}{lt_k}\Ex Y^2
\\ 
\label{ineq_1}
& \le lt_k\Pr(Y\ge kt_k/2) + (1+2\alpha)k^2l^{-1}t_k. 
\end{align}

By  Corollary \ref{byparts} we have (recall  definition\eqref{eq:defN})
\begin{align} 
\notag
\sum_{i=1}^n \Ex|X_i|\ind_{\{|X_i|\ge kt_k\}} 
& \le \Ex \max_{|I|=k}\sum_{i\in I}|X_i| 
+ \sum_{l=k+1}^\infty \left(kt_k\Pr(N(kt_k)\geq l)+\int_{kt_k}^\infty\Pr(N(s)\geq l)ds\right)
\\ 
\notag
& \le \frac 16 kt_k 
+\sum_{l=k+1}^\infty \left(kt_k \Ex N(kt_k)^2 l^{-2}+\int_{kt_k}^\infty \Ex N(s)^2 l^{-2}ds\right)
\\ 
\label{ineq_3}
& \le \frac16 kt_k +\frac1k \left(kt_k \Ex N(kt_k)^2+\int_{kt_k}^\infty \Ex N(s)^2 ds\right).
\end{align}
Assumption \eqref{cond_neg_cor} implies that
\begin{align*}
\Ex N(s)^2 
&= \sum_{i=1}^n \Pr(|X_i|\ge s) + \sum_{i\neq j} \Pr(|X_i|\ge s, |X_j|\ge s)
\\ 
&\le \sum_{i=1}^n \Pr(|X_i|\ge s) +  \alpha\left(\sum_{i=1}^n \Pr(|X_i|\ge s)\right)^2.
\end{align*}
Moreover for $s\ge kt_k$ we have 
\[
\sum_{i=1}^n \Pr(|X_i|\ge s) \le  \frac{1}{s}\sum_{i=1}^n \Ex|X_i| \ind_{\{|X_i|\ge s\}} \le \frac{kt_k}{s}\le 1,
\]
so
\[
\Ex N(s)^2 \le (1+\alpha)  \sum_{i=1}^n \Pr(|X_i|\ge s) \quad \mbox{ for }s\geq kt_k.
\]

Thus
\[
kt_k\Ex N(kt_k)^2 \le kt_k (1+\alpha) \sum_{i=1}^n\Pr (|X_i| \ge kt_k)
\le (1+\alpha)\sum_{i=1}^n \Ex|X_i|\ind_{\{|X_i|\ge kt_k\}},
\]
and
\[
\int_{kt_k}^\infty \Ex N(s)^2 ds 
\le  (1+\alpha )\sum_{i=1}^n\int_{kt_k}^\infty \Pr (|X_i|\ge s)ds 
\le (1+\alpha)\sum_{i=1}^n \Ex|X_i|\ind_{\{|X_i|\ge kt_k\}}.
\]
This together with \eqref{ineq_3} and the assumption that $k\ge 4(1+\alpha)$ implies
\[ 
\sum_{i=1}^n \Ex|X_i|\ind_{\{|X_i|\ge kt_k\}} \leq \frac{1}{3}kt_k
\]
and
\[
\Ex Y=\sum_{i=1}^n \Ex|X_i|\ind_{\{|X_i|\ge t_k\}}-\sum_{i=1}^n \Ex|X_i|\ind_{\{|X_i|\ge kt_k\}}\geq 
\frac{2}{3}kt_k.
\]
Therefore
\[
\Ex Y\ind_{\{Y\ge kt_k/2\}}\geq \Ex Y-\frac{1}{2}kt_k\geq \frac{1}{6}kt_k.
\]

This applied to \eqref{ineq_1} with $l= (12+24\alpha)k$ gives us $\Pr(Y\ge kt_k/2)\ge ( 144+288\alpha)^{-1}$ and in  consequence  
\[
\Ex \max_{|I|=k} \sum_{i\in I} |X_i| \ge \frac{1}{ 288(1+2\alpha)}kt(k,X).
\]

Since $k\mapsto kt(k,X)$ is non-decreasing, in the case $k\le \lceil 4(1+\alpha)\rceil =:k_0 \geq 8$ we have
\begin{align*}
\Ex \max_{|I|=k} |X_i| \ge \frac{k}{k_0} \Ex \max_{|I|=k_0} |X_i| 
\ge \frac{k}{5+4\alpha} \cdot \frac{1}{ 288(1+2\alpha)}k_0t(k_0,X) 
\\
\ge \frac{1}{ 36(5+4\alpha)(1+2\alpha)}kt(k,X). 		
\end{align*}

The last step is to loose the assumption that $X_i$ has no atoms. Note that both assumption \eqref{cond_neg_cor} and the lower bound depend only on $(|X_i|)_{i=1}^n$, so we may assume that $X_i$ are nonnegative almost surely.  Consider $X^{\eps}:=(X_i +\eps Y_i)_{i=1}^n$, where $Y_1,\ldots,Y_n$ 
are i.i.d. nonnegative r.v's with $\Ex Y_i <\infty $ and a   density $g$, independent of $X$.
Then for every $s,t >0 $ we have  (observe that \eqref{cond_neg_cor} holds also for $s<0$ or $t<0$).
\begin{align*}
\Pr(X_i^{\eps} \ge s, X_j^{\eps} \ge t) 
&= \int_{0}^{\infty} \int_{0}^{\infty} \Pr(X_i +\eps y_i \ge s, \ X_j +\eps y_j \ge t) g(y_i)g(y_j)dy_idy_j 
\\ 
&\mathop{\le}^{\eqref{cond_neg_cor}}
\alpha \int_{0}^{\infty} \int_{0}^{ \infty} \Pr(X_i \ge s - \eps y_i) \Pr(X_j \ge t- \eps y_j)
g(y_i)g(y_j)dy_idy_j 
\\
&  =
\alpha\Pr(X_i^{\eps} \ge s)\Pr( X_j^{\eps} \ge t).
\end{align*}
Thus $X^{\eps}$ satisfies  assumption \eqref{cond_neg_cor} 
and has the density function for every $\eps>0$. Therefore for all natural $k$ we have
\begin{align*}
\Ex \max_{|I|=k} \sum_{i=1}^n X_i^{\eps} \ge c(\alpha)kt(k, X^{\eps})  
\ge c(\alpha)kt(k, X).
\end{align*}
Clearly, $\Ex \max_{|I|=k} \sum_{i=1}^n X_i^{\eps}  \to \Ex \max_{|I|=k} \sum_{i=1}^n X_i$  as  $\eps\to 0$, so the lower bound holds in the case of arbitrary $X$ satisfying \eqref{cond_neg_cor}.	
\end{proof}

We may use Theorem  \ref{thm:neg_cor} to obtain a comparison of weak and strong moments for the supremum norm:
	
\begin{cor}
Let $X$ be  an $n$-dimensional centered random vector  satisfying condition \eqref{cond_neg_cor}. Assume that 	
\begin{equation}
\label{eq:2p_to_p}
\|X_i\|_{2p} \le \beta \|X_i\|_p \qquad \mbox{for every $p\ge 2$ and $i=1,\ldots,n$}.
\end{equation}
Then the following comparison of weak and strong moments for the supremum norm holds: 
for all $a\in \er^n$ and all $p\ge 1$,
\[
\bigl(\Ex \max_{i\le n}|a_iX_i|^p\bigr)^{1/p} 
\le C(\alpha, \beta)\Bigl[ \Ex \max_{i\le n}|a_iX_i| + \max_{i\le n}\bigl( \Ex |a_iX_i|^p \bigr)^{1/p}\Bigr],
\]
where $C(\alpha, \beta)$ is a constant depending only on $\alpha$ and $\beta$.
\end{cor}

\begin{proof}
Let $X'=(X_i')_{i\leq n}$ be a decoupled version of $X$. For any $p>0$ a random vector $(|a_iX_i|^p)_{i\le n}$  satisfies condition \eqref{cond_neg_cor}, so by Theorem \ref{thm:neg_cor} 
\[
\bigl(\Ex \max_{i\le n}|a_iX_i|^p\bigr)^{1/p} \sim \bigl(\Ex \max_{i\le n}|a_iX_i'|^p\bigr)^{1/p} 
\]
for all $p>0$, up to a constant depending only on $\alpha$.
The coordinates of $X'$ are independent and satisfy  condition \eqref{eq:2p_to_p}, so due to \cite[Theorem 1.1]{LS} the comparison of weak and strong moments of $X'$ holds, i.e.  for $p\geq 1$,
\[
\bigl(\Ex \max_{i\le n}|a_iX_i'|^p\bigr)^{1/p} 
\le C(\beta)\Bigl[ \Ex \max_{i\le n}|a_iX_i'| + \max_{i\le n}\bigl( \Ex |a_iX_i'|^p \bigr)^{1/p}\Bigr],
\]
where $C(\beta)$ depends only on $\beta$. These two observations  yield the assertion.
\end{proof}

\section{Lower estimates for order statistics} \label{sect:order_statistics}

The next lemma shows the relation between $t(k,X)$ and $t^*(k, X)$ for log-concave vectors $X$.

\begin{lem}
\label{lem:t_vs_t*}
Let $X$ be a symmetric log-concave random vector in $\er^n$.
For any $1\leq k \leq n$ we have
\[
\frac{1}{3}\left(t^*(k,X)+\frac{1}{k}\max_{|I|=k}\sum_{i\in I}\Ex |X_i|\right)
\leq t(k,X)
\leq  4\left(t^*(k,X)+\frac{1}{k}\max_{|I|=k}\sum_{i\in I}\Ex |X_i|\right).
\]
\end{lem}

\begin{proof}
Let $t_k:=t(k,X)$ and $t_k^*:=t^*(k,X)$. We may assume that any $X_i$ is not identically equal to $0$. 
Then $\sum_{i=1}^n \Pr(|X_i|\ge t_k^{ *})= k$ and 
$\sum_{i=1}^n \Ex|X_i|_{{ \{}|X_i|\ge t_k{ \}}}=kt_k$.

Obviously $t_k^*\leq t_k$. Also for any $|I|=k$ we have
\[
\sum_{i\in I}\Ex |X_i|\leq \sum_{i\in I}\left(t_k+\Ex |X_i|\ind_{\{|X_i|\geq t_k\}}\right)
{ \leq} |I|t_k+kt_k=2kt_k.
\]

To prove the upper bound set
\[
I_1:=\{i\in [n]\colon\ \Pr(|X_i|\geq t_k^*)\geq 1/4\}.
\]
We have
\[
k  \ge\sum_{i\in |I_1|}\Pr(|X_i|\geq t_k^*)\geq \frac{1}{4}|I_1|,
\]
so $|I_1|\leq 4k$. Hence
\[
\sum_{i\in I_1}\Ex |X_i|\ind_{\{|X_i|\geq t_k^*\}}\leq \sum_{i\in I_1}\Ex |X_i|\leq
4 \max_{|I|=k}\sum_{i\in I}\Ex |X_i|.
\]

Moreover by the second part of Lemma \ref{lem:cutmean} we get
\[
\Ex|X_i|\ind_{\{|X_i|\geq t_k^*\}}\leq 4t_k^*\Pr(|X_i|\geq t_k^*) \quad\mbox{for }i\notin I_1,
\]
so  
\[
\sum_{i\notin I_1}\Ex|X_i|\ind_{\{|X_i|\geq t_k^*\}}\leq 4t_k^*\sum_{i=1}^n\Pr(|X_i|\geq t_k^*)\leq 4kt_k^*.
\]
Hence if $s=4t_k^*+\frac{4}{k}\max_{|I|=k}\sum_{i\in I}\Ex|X_i|$ then
\[
\sum_{i=1}^n\Ex|X_i|\ind_{\{|X_i|\geq s\}}
\leq \sum_{i=1}^n\Ex|X_i|\ind_{\{|X_i|\geq t_k^*\}}\leq 4 \max_{|I|=k}\sum_{i\in I}\Ex |X_i|+4kt_k^*=ks,
\]
that is $t_k\leq s$.
\end{proof}

To derive bounds for order statistics we will also need a few facts about log-concave vectors.

\begin{lem}
\label{lem:density}
Assume that $Z$ is an isotropic one- or two-dimensional  log-concave random  vector with  
a density $g$. Then $g(t)\le C$ for all $t$. If $Z$ is one-dimensional, then also $g(t)\ge c$ for all $|t|\le t_0$, where $t_0>0$  is an absolute constant. 
\end{lem}

\begin{proof}
We will use a classical result (see \cite[Theorem 2.2.2, Proposition 3.3.1 and Proposition 2.5.9]{BGVV}):  
$\|g\|_{\sup}\sim g(0)  \sim 1$ (note that here we use the assumption that $Z$ is isotropic, in particular that $\Ex Z=0$, and   that the dimension of $Z$ is $1$ or $2$). 
This implies the upper bound on $g$. 

In order to get the lower bound in the one-dimensional case, it suffices to prove that $g(u)\ge c$ for 
$|u|=\eps \Ex |Z|\ge (2C_1)^{-1} \eps $, where $1/4>\eps>0$ is  fixed  and its value will be chosen later (then by the log-concavity we get $g(u)^sg(0)^{1-s}\le g(su)$ for all $s\in (0,1)$). Since $-Z$ is again isotropic we may assume that $u\geq 0$.

If  $g(u) \ge g(0)/e$, then we are done. Otherwise by log-concavity of $g$ we get
\[
\Pr(Z\ge u) = \int_u^\infty g(s) ds 
\le  \int_u^\infty g(u)^{s/u}g(0)^{-s/u +1} ds \le  g(0) \int_u^\infty e^{-s/u}du \le C_0u
 \leq C_0\ve.
\]
On the other hand, $Z$ has mean zero, so $\Ex |Z|=2\Ex Z_{+}$ and
by the Paley--Zygmund inequality and \eqref{eq:reg_mom} we have 
\[
\Pr(Z\ge u) =  \Pr(Z_+\geq 2\eps \Ex Z_{+})
\ge (1-2\eps)^2 \frac{(\Ex Z_+)^2}{\Ex Z_+^2} 
\ge \frac{1}{16}\frac{(\Ex |Z|)^2}{\Ex Z^2}\geq c_0.
\]
For $\eps<c_0/C_0$ we get a contradiction.
\end{proof}

\begin{lem}
\label{lem:tail1}
Let $Y$ be a mean zero log-concave random variable and let $\Pr(|Y|\geq t)\leq p$ for some $p>0$. Then
\[
\Pr\left(|Y|\geq \frac{t}{2}\right)\geq \frac{1}{\sqrt{ep}}\Pr(|Y|\geq t).
\]
\end{lem}

\begin{proof}
By the Gr\"unbaum inequality \eqref{Grunbaum} we have $\Pr(Y\geq 0)\geq 1/e$, hence
\[
\Pr\left(Y\geq \frac{t}{2}\right)\geq \sqrt{\Pr(Y\geq t)\Pr(Y\geq 0)}
\geq \frac{1}{\sqrt{e}}\sqrt{\Pr(Y\geq t)}
\geq \frac{1}{\sqrt{ep}}\Pr(Y\geq t).
\]
Since $-Y$ satisfies the same assumptions as $Y$ we also have
\[
\Pr\left(-Y\geq \frac{t}{2}\right)
\geq \frac{1}{\sqrt{ep}}\Pr(-Y\geq t). \qedhere
\]
\end{proof}

\begin{lem}
\label{lem:tail2}
Let $Y$ be a mean zero log-concave random variable and let $\Pr(|Y|\geq t)\geq p$ for some $p>0$. Then there
exists a  universal constant $C$ such that
\[
\Pr(|Y|\leq \lambda t)\leq \frac{C\lambda}{\sqrt p} \Pr(|Y|\geq t)\quad \mbox{for }\lambda\in [0,1].
\]
\end{lem}

\begin{proof}
Without loss of generality we may assume that $\Ex Y^2=1$. Then by Chebyshev's inequality $t\leq p^{-1/2}$. Let $g$ be the density of
$Y$. By Lemma \ref{lem:density} we know that $\|g\|_\infty\leq C$ and $g(t)\geq c$ on $[-t_0,t_0]$, where $c,C$ and $t_0\in (0,1)$ are universal constants.
Thus
\[
\Pr(|Y|\leq t)\geq \Pr(|Y|\leq t_0\sqrt{p}t)\geq 2ct_0\sqrt{p}t,
\]
and
\[
\Pr(|Y|\leq \lambda t)\leq 2\|g\|_\infty\lambda t\leq 2C\lambda t\leq \frac{C\lambda}{ct_0\sqrt{p}}\Pr(|Y|\leq t).\qedhere
\]
\end{proof}

Now we are ready to give a proof of {the lower bound in Theorem \ref{thm:estkmax}}. 
The next proposition is a key part of it.

\begin{prop}
\label{prop:kmaxmean}
Let $X$ be a mean zero log-concave $n$-dimensional random vector with uncorrelated coordinates and let 
$\alpha>1/4$.
Suppose that
\[
\Pr\bigl(|X_i|\geq t^*(\alpha,X)\bigr)\leq \frac{1}{C_3}\quad \mbox{for all }i.
\]
Then
\[
\Pr\Bigl(\lfloor 4\alpha\rfloor \text{-}\max_{i}|X_i|\geq \frac{1}{C_4} t^*(\alpha,X)\Bigr)\geq \frac{3}{4}. 
\]
\end{prop}

\begin{proof}
Let $t^*=t^*(\alpha,X)$, $k:=\lfloor 4\alpha\rfloor$ and 
$L=\lfloor \frac{\sqrt{C_3}}{4 \sqrt{e}}\rfloor$. 
We will choose $C_3$ in such a way that $L$ is large, in particular we may assume that $L\geq 2$. 
 Observe also that  $\alpha = \sum_{i=1}^n \Pr(|X_i|\geq t^*(\alpha,X))\leq nC_3^{-1}$, thus $Lk\leq C_3^{1/2}e^{-1/2}\alpha\leq e^{-1/2}C_3^{-1/2}n\leq n$  if $C_3\geq 1$.
Hence
\begin{equation}
\label{eq:kmax1}
k\text{-}\max_{i}|X_i|\geq \frac{1}{k(L-1)}\sum_{l=k+1}^{Lk}l\text{-}\max_{i}|X_i|
=\frac{1}{k(L-1)}\biggl(\max_{|I|=Lk}\sum_{i\in I}|X_i|-\max_{|I|=k}\sum_{i\in I}|X_i|\biggr).
\end{equation}

Lemma \ref{lem:tail1} and the definition of $t^*(\alpha,X)$ yield
\[
\sum_{i=1}^n\Pr\left(|X_i|\geq \frac{1}{2}t^*\right)\geq \frac{\sqrt{C_3}}{ \sqrt{e}}\alpha \geq Lk.
\]
This yields $t(Lk,X)\geq t^*(Lk,X) \geq \frac{t^*}{2}$ and by Theorem \ref{thm:kmaxlogconc} we have
\[
\Ex\max_{|I|=Lk}\sum_{i\in I}|X_i|\geq c_1Lk\frac{t^*}{2}.
\]
Since for any norm $\Pr(\|X\|\leq t\Ex \|X\|)\leq Ct$ for $t>0$ (see \cite[Corollary 1]{La}) we have
\begin{equation}
\label{eq:kmax2}
\Pr\left(\max_{|I|=Lk}\sum_{i\in I}|X_i|\geq c_2Lkt^*\right)\geq \frac{7}{8}.
\end{equation}

By the Paley-Zygmund inequality and \eqref{eq:reg_mom},
$\Pr(|X_i|\geq \frac{1}{2}\Ex |X_i|)\geq \frac{(\Ex|X_i|)^2}{4\Ex |X_i|^2}> \frac{1}{C_3}$ 
if $C_3 >4C_1^2$, so $\frac{1}{2}\Ex|X_i|\leq t^*$. Moreover it is easy to verify that
$k=\lfloor 4\alpha\rfloor>\alpha$ for $\alpha>1/4$, thus $t^*(k,X)\leq t^*(\alpha,X)=t^*$. Hence 
 Proposition \ref{lem_easy} and
Lemma  \ref{lem:t_vs_t*} yield
\[
\Ex \max_{|I|=k}\sum_{i\in I}|X_i|\leq 2t(k,X)\leq 8\bigl(t^*(k,X)+\max_{i}\Ex |X_i|\bigr)\leq 24t^*,
\]
and therefore
\begin{equation}
\label{eq:kmax3}
\Pr\left(\max_{|I|=k}\sum_{i\in I}|X_i|\geq 200kt^*\right)\leq \frac{1}{8}.
\end{equation}

Estimates  \eqref{eq:kmax1}-\eqref{eq:kmax3} yield
\[
\Pr\left(k\text{-}\max_{i}|X_i|\geq \frac{1}{L-1}(c_2L-200)t^*\right)\geq \frac{3}{4},
\]
so it is enough to choose $C_3$ in such a way that $L\geq 400/c_2$.
\end{proof}

\begin{proof}[Proof of the first part of Theorem \ref{thm:estkmax}]
Let $t^*=t^*(k-1/2,X)$ and $C_3$ be as in Proposition \ref{prop:kmaxmean}. 
It is enough to consider the case when $t^*>0$, then $\Pr(|X_i|=t^*)=0$ for all $i$ and 
$\sum_{i=1}^n \Pr(|X_i|\geq t^*) = k-1/2$.
Define
\[
I_1:=\left\{i\leq n\colon\ \Pr(|X_i|\geq t^*)\leq \frac{1}{C_3}\right\},\quad
\alpha:=\sum_{i\in I_1}\Pr(|X_i|\geq t^*),
\]
\[
I_2:=\left\{i\leq n\colon\ \Pr(|X_i|\geq t^*)> \frac{1}{C_3}\right\},\quad
\beta:=\sum_{i\in I_2}\Pr(|X_i|\geq t^*).
\]

If $\beta=0$ then $\alpha=k-1/2$, $|I_1|=[n]$, and the assertion immediately follows by Proposition \ref{prop:kmaxmean}  
since $4\alpha\geq k$.

Otherwise define
\[
\tilde{N}(t):=\sum_{i\in I_2}\ind_{\{|X_i|\leq t\}}.
\]
We have by Lemma \ref{lem:tail2} applied with $p=1/C_3$
\[
\Ex \tilde{N}(\lambda t^*)=\sum_{i\in I_2}\Pr(|X_i|\leq \lambda t^*)
\leq C_5\lambda\sum_{i\in I_2}\Pr(|X_i|\leq t^*)=C_5\lambda(|I_2|-\beta).
\]
Thus
\[
\Pr\left(\lceil \beta \rceil \text{-}\max_{i\in I_2}|X_i|\leq \lambda t^*\right)
= \Pr(\tilde{N}(\lambda t^*)\geq |I_2|+1-\lceil \beta \rceil)\leq
\frac{1}{|I_2|+1-\lceil \beta \rceil}\Ex \tilde{N}(\lambda t^*)\leq C_5\lambda.
\]
Therefore
\[
\Pr\Bigl(\lceil \beta \rceil \text{-}\max_{i\in I_2}|X_i|\geq \frac{1}{4C_5} t^*\Bigr)\geq \frac{3}{4}.
\]
If $\alpha<1/2$ then $\lceil \beta \rceil =k$ and the assertion easily follows. Otherwise 
Proposition \ref{prop:kmaxmean} yields
\[
\Pr\Bigl(\lfloor 4\alpha\rfloor \text{-}\max_{i\in I_1}|X_i|\geq \frac{1}{C_4}t^*\Bigr)\geq \frac{3}{4}.
\]
Observe that for $\alpha\geq 1/2$ we have $\lfloor 4\alpha\rfloor+\lceil \beta \rceil\geq 4\alpha-1+\beta
\geq \alpha+1/2+\beta=k$, so

\begin{align*}
\Pr\left(k\text{-}\max_{i}|X_i|\geq \min\left\{\frac{t^*}{C_4},\frac{t^*}{4C_5}\right\}\right)
&\geq
\Pr\left(\lfloor 4\alpha\rfloor\text{-}\max_{i\in I_1}|X_i|\geq \frac{1}{C_4}t,
\lceil \beta \rceil \text{-}\max_{i\in I_2}|X_i|\geq \frac{1}{4C_5} t^*\right)
\\
&\geq \frac{1}{2}. \qedhere
\end{align*}
\end{proof}

\begin{rmk}
A modification of the proof above shows that under the assumptions of Theorem \ref{thm:estkmax} for any $p<1$ there exists $c(p)>0$ such that
\[
\Pr\left( k\text{-}\max_{i\leq n}|X_i|\geq c(p)t^*(k-1/2,X)\right)\geq p.
\]
\end{rmk}

\section{Upper estimates for order statistics} \label{sect:order_statistics_upper}

We will need a few more facts  concerning log-concave vectors.

\begin{lem}
\label{lem:corr}
Suppose that $X$ is a mean zero log-concave random vector with uncorrelated coordinates. 
Then for any $i\neq j$ and $s>0$,
\[
\Pr(|X_i|\leq s,|X_j|\leq s)\leq C_6\Pr(|X_i|\leq s)\Pr(|X_j|\leq s).
\]
\end{lem}

\begin{proof}
Let $C_7,c_3$ and $t_0$ be the constants from Lemma \ref{lem:density}.
If $s>t_0\|X_i\|_2$ then, by Lemma \ref{lem:density}, $\Pr(|X_i|\leq s)\geq 2c_3t_0$ and the assertion is obvious (with any $C_6\geq (2c_3t_0)^{-1}$).
Thus we will assume that $s\leq t_0\min\{\|X_i\|_2,\|X_j\|_2\}$.

Let $\widetilde{X}_i=X_i/\|X_i\|_2$ and let $g_{ij}$ be the density of $(\widetilde{X}_i,\widetilde{X}_j)$. 
By Lemma \ref{lem:density} we know
that $\|g_{i,j}\|_{\infty}\leq C_7$, so 
\[
\Pr(|X_i|\leq s,|X_j|\leq s)
=\Pr(|\widetilde{X}_i|\leq s/\|X_i\|_2,|\tilde{X}_j|\leq s/\|X_j\|_2)\leq 
C_7\frac{s^2}{\|X_i\|_2\|X_j\|_2}.
\] 
On the other hand the second part of Lemma \ref{lem:density} yields
\[
\Pr(|X_i|\leq s)\Pr(|X_j|\leq s)
\geq \frac{4c_3^2s^2}{\|X_i\|_2\|X_j\|_2}. \qedhere
\]
\end{proof}

\begin{lem}
\label{lem:dil1}
Let $Y$ be a log-concave random variable. Then
\[
\Pr(|Y|\geq ut)\leq \Pr(|Y|\geq t)^{(u-1)/2}\quad \mbox{for } u\geq 1,t\geq 0.
\]
\end{lem}

\begin{proof}
We may assume that $Y$ is non-degenerate (otherwise the statement is obvious), in particular
$Y$ has no atoms. Log-concavity of $Y$ yields
\[
\Pr(Y\geq t)\geq \Pr(Y\geq -t)^{\frac{u-1}{u+1}}\Pr(Y\geq ut)^{\frac{2}{u+1}}.
\]
Hence
\begin{align*}
\Pr(Y\geq ut)
&\leq \left(\frac{\Pr(Y\geq t)}{\Pr(Y\geq -t)}\right)^{\frac{u+1}{2}}\Pr(Y\geq -t)
=\left(1-\frac{\Pr(|Y|\leq t)}{\Pr(Y\geq -t)}\right)^{\frac{u+1}{2}}\Pr(Y\geq -t)
\\
&\leq (1-\Pr(|Y|\leq t))^{\frac{u+1}{2}}\Pr(Y\geq -t)=\Pr(|Y|\geq t)^{\frac{u+1}{2}}\Pr(Y\geq -t).
\end{align*}
Since $-Y$ satisfies the same assumptions as $Y$, we also have
\[
\Pr(Y\leq -ut)\leq \Pr(|Y|\geq t)^{\frac{u+1}{2}}\Pr(Y\leq t).
\]
Adding both estimates we get
\[
\Pr(|Y|\geq ut)\leq \Pr(|Y|\geq t)^{\frac{u+1}{2}}(1+\Pr(|Y|\leq t))
=\Pr(|Y|\geq t)^{\frac{u-1}{2}}(1-\Pr(|Y|\leq t)^2).\qedhere
\]
\end{proof}

\begin{lem}
\label{lem:dilation}
Suppose that $Y$ is a log-concave   random variable  and $\Pr(|Y|\leq t)\leq \frac{1}{10}$. 
Then $\Pr(|Y|\leq 21t)\geq 5\Pr(|Y|\leq t)$.
\end{lem}

\begin{proof}
Let $\Pr(|Y|\leq t)=p$ then by Lemma \ref{lem:dil1}
\[
\Pr(|Y|\leq 21t)=1-\Pr(|Y|>21t)\geq 1-\Pr(|Y|>t)^{10}=1-(1-p)^{10}\geq 10p-45p^2\geq 5p. \qedhere
\]
\end{proof}

Let us now prove \eqref{eq:kmaxupgen1} and see how it implies the second part of Theorem \ref{thm:estkmax}. Then we give a proof of  \eqref{eq:kmaxupgen2}.

\begin{proof}[Proof of \eqref{eq:kmaxupgen1}]
Fix $k$ and set $t^*:=t^* (k-1/2 , X )$. Then $\sum_{i=1}^n \Pr(|X_i|\geq t^*)=k-1/2$. Define
\begin{align}
\label{eq:defI1}
I_1:=\left\{i\leq n\colon\ \Pr(|X_i|\geq t^*)\leq \frac{9}{10}\right\},\quad
\alpha:=\sum_{i\in I_1}\Pr(|X_i|\geq t^*),
\\
\label{eq:defI2}
I_2:=\left\{i\leq n\colon\ \Pr(|X_i|\geq t^*)> \frac{9}{10}\right\},\quad
\beta:=\sum_{i\in I_2}\Pr(|X_i|\geq t^*).
\end{align}
Observe that for $u>3$ and $1\leq l\le |I_1|$ we have by Lemma \ref{lem:dil1}
\begin{align}
\label{eq:estI1}
\Pr(l\text{-}\max_{i\in I_1}|X_i|\geq ut^*)
&\leq \Ex\frac{1}{l}\sum_{i\in I_1}\ind_{\{|X_i|\geq ut^*\}}
=\frac{1}{l}\sum_{i\in I_1}\Pr(|X_i|\geq ut^*)
\\
\notag
&\leq \frac{1}{l}\sum_{i\in I_1}\Pr(|X_i|\geq t^*)^{(u-1)/2}
\leq \frac{\alpha}{l}\left(\frac{9}{10}\right)^{(u-3)/2}.
\end{align}

Consider two cases.

\textbf{Case 1.}\ $\beta>|I_2|-1/2$. Then $|I_2|<\beta+1/2\leq k$, so $ k-|I_2|\ge 1$ and
\[
\alpha=k-\frac{1}{2}-\beta \le k - |I_2|. 
\]
Therefore by \eqref{eq:estI1}
\[
\Pr\left(k\text{-}\max |X_i|\geq 5t^*\right)
\leq 
\Pr\left((k-|I_2|)\text{-}\max_{i\in I_1} |X_i|\geq 5t^*\right)
\leq \frac{9}{10}.
\]

\textbf{Case 2.}\ $\beta\le|I_2|-1/2$. 
Observe that for any disjoint sets $J_1$, $J_2$ and integers $l, m$ such that $l\le |J_1|$,  $m \le |J_2|$ we have
\begin{equation}\label{eq:split_sets_k-max}
(l+m-1)\text{-}\max_{i\in J_1\cup J_2}|x_i|
\leq \max\left\{l\text{-}\max_{i\in J_1}|x_i|,m\text{-}\max_{i\in J_2}|x_i|\right\}
\leq l\text{-}\max_{i\in J_1}|x_i|+m\text{-}\max_{i\in J_2}|x_i|.
\end{equation}
Since
\[
\lceil \alpha\rceil +\lceil\beta\rceil \leq \alpha+\beta+2<k+2
\]
we have $\lceil \alpha\rceil +\lceil\beta\rceil\leq k+1$ and, by \eqref{eq:split_sets_k-max},
\[
k\text{-}\max_{i}|X_i|
\leq  \lceil \alpha\rceil\text{-}\max_{i\in I_1}|X_i|+
\lceil \beta\rceil\text{-}\max_{i\in I_2}|X_i|.
\]
Estimate \eqref{eq:estI1} yields 
\[
\Pr\left( \lceil \alpha\rceil\text{-}\max_{i\in I_1}|X_i|\geq ut^*\right)\leq 
\left(\frac{9}{10}\right)^{(u-3)/2} \quad \mbox{for }u\geq 3.
\]

To estimate $\lceil \beta\rceil\text{-}\max_{i\in I_2}|X_i|
=(|I_2|+1-\lceil\beta\rceil)\text{-}\min_{i\in I_2}|X_i|$ observe that
by Lemma \ref{lem:dilation}, the definition of $I_2$ and assumptions on $\beta$,
\[
\sum_{i\in I_2}\Pr(|X_i|\leq 21t^*)\geq 5\sum_{i\in I_2}\Pr(|X_i|\leq t^*)=
5(|I_2|-\beta)\geq 2(|I_2|+1-\lceil \beta\rceil).
\]

Set $l:=(|I_2|+1-\lceil \beta\rceil)$ and 
\[
\tilde{N}(t):=\sum_{i\in I_2}\ind_{{\{}|X_i|\leq t{\}}}.
\]
Note that we know already that $\Ex\tilde{N}(21 t^*) \ge 2l$. Thus 
the Paley-Zygmund inequality implies
\begin{align*}
\Pr\left(\lceil \beta\rceil\text{-}\max_{i\in I_2}|X_i|\leq 21t^*\right)
&=\Pr\left(l\text{-}\min_{i\in I_2}|X_i|\leq 21t^*\right)
\geq \Pr( \tilde{N}(21t^*)\geq l)
\\
&\geq \Pr\left( \tilde{N}(21t^*)\geq \frac{1}{2}\Ex\tilde{N}(21t^*)\right)
\geq \frac{1}{4}\frac{(\Ex\tilde{N}(21t^*))^2}{\Ex\tilde{N}(21t^*)^2}.
\end{align*}
However Lemma \ref{lem:corr}  yields
\[
\Ex\tilde{N}(21t^*)^2\leq \Ex\tilde{N}(21t^*)+C_6(\Ex\tilde{N}(21t^*)))^2\leq
(C_6+1)(\Ex\tilde{N}(21t^*))^2.
\]
Therefore 
\begin{align*}
\Pr\left(k\text{-}\max_{i}|X_i|> (21+u)t^*\right)
&\leq
\Pr\left( \lceil \alpha\rceil\text{-}\max_{i\in I_1}|X_i|\geq ut^*\right)+
\Pr\left(\lceil \beta\rceil\text{-}\max_{i\in I_2}|X_i|> 21t^*\right)
\\
&\leq
\left(\frac{9}{10}\right)^{(u-3)/2}+1-\frac{1}{4(C_6+1)}\leq 1-\frac{1}{5(C_6+1)}
\end{align*}
for sufficiently large $u$.
\end{proof}

The unconditionality assumption plays a crucial role in the proof of the next lemma, which
allows to derive the second part of Theorem \ref{thm:estkmax} from
estimate \eqref{eq:kmaxupgen1}.

\begin{lem}
\label{lem:tailkmin}
Let $X$ be an unconditional log-concave $n$-dimensional random vector. Then for any $1\leq k\leq n$,
\[
\Pr\left(k\text{-}\max_{i\leq n}|X_i|\geq ut\right)\leq \Pr\left(k\text{-}\max_{i\leq n}|X_i|\geq t\right)^u
\quad \mbox{ for }u>1,t>0.
\]
\end{lem}

\begin{proof}
Let $\nu$ be the law of $(|X_1|,\ldots,|X_n|)$. Then $\nu$ is log-concave on $\er_n^+$.
Define for $t>0$,
\[
A_t:=\left\{x\in \er_n^+\colon\ k\text{-}\max_{i\leq n}|x_i|\geq t\right\}.
\]
It is easy to check that $\frac{1}{u}A_{ut}+(1-\frac{1}{u})\er_+^n\subset A_t$, hence
\[
\Pr\left(k\text{-}\max_{i\leq n}|X_i|\geq t\right)=\nu(A_t)\geq \nu(A_{ut})^{1/u}\nu(\er_+^n)^{1-1/u}
=\Pr\left(k\text{-}\max_{i\leq n}|X_i|\geq ut\right)^{1/u}. \qedhere
\]
\end{proof}

\begin{proof}[Proof of the second part of Theorem \ref{thm:estkmax}]
Estimate \eqref{eq:kmaxupgen1} together with Lemma \ref{lem:tailkmin} yields
\[
\Pr\left(k\text{-}\max_{i\leq n}|X_i|\geq Cut^*(k-1/2.X)\right)
\leq (1-c)^u 
\quad \text{for }u\ge 1,
\]
and the assertion follows by integration by parts.
\end{proof}

\begin{proof}[Proof of \eqref{eq:kmaxupgen2}]
Define $I_1$, $I_2$, $\alpha$ and $\beta$ by \eqref{eq:defI1} and \eqref{eq:defI2},
where this time $t^*=t^*(k-k^{5/6}/2,X)$. 
Estimate \eqref{eq:estI1} is still valid so integration by parts yields 
\[
\Ex l\text{-}\max_{i\in I_1}|X_i|\leq \left(3+20\frac{\alpha}{l}\right)t^*.
\]
Set
\[
k_{\beta}:=\left\lceil\beta+\frac{1}{2}k^{5/6}\right\rceil.
\]
Observe that
\[
\lceil\alpha\rceil+k_\beta<
\alpha+\beta+\frac{1}{2}k^{5/6}+2=k+2.
\]
Hence $\lceil\alpha\rceil+k_\beta\leq k+1$.

If $k_\beta > |I_2|$, then $k-|I_2| \ge \lceil\alpha\rceil+k_\beta -1 - |I_2| \ge \lceil\alpha\rceil$, so 
\[
\Ex  k\text{-}\max_{i}|X_i|\leq \Ex  (k-|I_2)\text{-}\max_{i\in I_1}|X_i|
\leq  \Ex  \lceil\alpha\rceil\text{-}\max_{i\in I_1}|X_i| \leq 23t^*.
\]
Therefore it suffices to consider case $k_\beta \le |I_2|$ only.

Since $\lceil\alpha\rceil+k_\beta -1 \leq k$ and $k_\beta \le |I_2|$, we have by \eqref{eq:split_sets_k-max},
\[
\Ex k\text{-}\max_{i}|X_i|\leq \Ex \lceil\alpha\rceil\text{-}\max_{i\in I_1}|X_i|
+\Ex k_\beta\text{-}\max_{i\in I_2}|X_i|
\leq 23t_*+\Ex k_\beta\text{-}\max_{i\in I_2}|X_i|.
\] 
Since $\beta\leq k-\frac{1}{2}k^{5/6}$ and $x\rightarrow x-\frac{1}{2}x^{5/6}$ is increasing for
$x\geq 1/2$ we have
\[
\beta\leq \beta+\frac{1}{2}k^{5/6}-\frac{1}{2}\left(\beta+\frac{1}{2}k^{5/6}\right)^{5/6}
\leq k_\beta-\frac{1}{2}k_\beta^{5/6}.
\]

Therefore, considering $(X_{i})_{i\in I_2}$ instead of $X_i$ and $k_\beta$ instead of $k$
it is enough to show the following claim: \\
Let $s>0$, $n\ge k$ and let $X$ be an  $n$-dimensional log-concave vector. 
 Suppose that 
\[
\sum_{i\leq n}\Pr(|X_i|\geq s)\leq k-\frac{1}{2}k^{5/6}\quad \mbox{ and }
\quad \min_{i\leq n}\Pr(|X_i|\geq s)\geq 9/10 
\]
then
\[
\Ex k\text{-}\max_{i\leq n}|X_i|\leq C_8s.
\]

We will show the claim by induction on $k$. For $k=1$ the statement is obvious (since the assumptions
are contradictory). Suppose now that $k\geq 2$ and the assertion holds for $k-1$.

{\bf Case 1.} $\Pr(|X_{i_0}|\geq s)\geq 1-\frac{5}{12}k^{-1/6}$ for some $1\leq i_0\leq n$.
Then
\[
\sum_{i\neq i_0}\Pr(|X_i|\geq s)\leq k-\frac{1}{2}k^{5/6}-\left(1-\frac{5}{12}k^{-1/6}\right)
\leq k-1-\frac{1}{2}(k-1)^{5/6}, 
\]
where to get the last inequality we used that $x^{5/6}$ is concave on $\er_+$, so 
$(1-t)^{5/6}\leq 1-\frac{5}{6}t$ for $t=1/k$.
Therefore by the induction assumption applied to $(X_i)_{i\neq i_0}$,
\[
\Ex k\text{-}\max_{i}|X_i|\leq \Ex (k-1)\text{-}\max_{i\neq i_0}|X_i|\leq C_8s.
\]

{\bf Case 2.} $\Pr(|X_{i}|\leq s)\geq \frac{5}{12}k^{-1/6}$ for all $i$. 
Applying Lemma \ref{lem:density} we get 
\[
\frac{5}{12}k^{-1/6} \leq \Pr\left(\frac{|X_i|}{\|X_i\|_2}\leq \frac{s}{\|X_i\|_2}\right)
\le C\frac{s}{\|X_i\|_2},
\]
so $\max_i\|X_i\|_2\leq Ck^{1/6}s$. Moreover $n\leq \frac{10}{9}k$.
Therefore by the result of Lee and Vempala \cite{LV} $X$ satisfies the exponential
concentration with $\alpha\leq C_9k^{5/12}s$.

Let $l=\lceil k-\frac{1}{2}(k^{5/6}-1)\rceil$ then $s\geq t_*(l-1/2, X)$ and 
$k-l+1\geq \frac{1}{2}(k^{5/6}-1) \geq \frac{1}{9}k^{5/6}$.
Let 
\[
A:=\left\{x\in \er^n\colon l\text{-}\max_{i}|{ x}_i|\leq C_{10}s \right\}.
\]
By \eqref{eq:kmaxupgen1} (applied with $l$ instead of $k$) 
we have $\Pr(X\in A)\geq c_{4}$. 
Observe that 
\[
k\text{-}\max_{i}|x_i|\geq C_{10}s+u\Rightarrow
\mathrm{dist}(x,A)\geq \sqrt{k-l+1}u\geq  \frac{1}{3}k^{5/12}u.
\]
Therefore by Lemma \ref{lem:koncsmallmeas} we get
\[
\Pr\left(k\text{-}\max_{i}|X_i|\geq C_{10}s+3C_9us\right)\leq
\exp\left(-(u+\ln c_{4})_{+}\right).
\]
Integration by parts yields
\[
\Ex k\text{-}\max_{i}|X_i|\leq \left(C_{10}+3C_9(1-\ln c_{4})\right)s
\]
and the induction step is shown in this case provided that
$C_8\geq C_{10}+3C_9(1-\ln c_{4})$.
\end{proof}

 To obtain Corollary  \ref{cor:estkmaxisotr} we used the following lemma.

\begin{lem}\label{lem:asympt_tk}
Assume that $X$ is a symmetric isotropic log-concave vector in $\er^n$. Then 
\begin{equation}
\label{eq:estt*1} 
t^*( p,X)\sim \frac{ n- p}{n}  \quad \text{for }  n >  p\geq n/4.
\end{equation}
and
\begin{equation}
\label{eq:estt^*2}
t^*( k/2,X)\sim t^*(k,X)\sim t(k,X)  \quad \text{for } k\leq n/2.
\end{equation}
\end{lem}

\begin{proof}
Observe that
\[
\sum_{i=1}^n\Pr(|X_i|\leq t^*( p,X))=n- p.
\]
Thus Lemma   \ref{lem:density} implies that for $ p\ge  c_5n$ (with $ c_5\in (\frac 12,1)$) we have 
$t^*( p,X)\sim \frac{n- p}{n}$. Moreover, by the Markov inequality
\[
\sum_{i=1}^n \Pr(|X_i|\ge  4) \le 	\frac n 4,
\]
so $t^*(n/ 4,X) \le 4$. Since $ p\mapsto t^*(p,X)$ is non-increasing, 
we know that $t^*( p,X)\sim 1$ for $n/ 4\le  p\le  c_5 n$. 

Now we will prove \eqref{eq:estt^*2}.  We have
\[
t^*(k,X)\leq t^*( k/2,X)\leq t( k/2,X)\leq 2t(k,X),
\]
so it suffices to show that $ t^*(k,X)\ge ct(k,X)$. To this end we fix $k\le n/2$.
 By \eqref{eq:estt*1} we know that $t:=C_{11}t^*(k,X)\geq C_{11}t^*(n/2,X)\geq e$, so the isotropicity of $X$ and Markov's inequality yield
$\Pr(|X_i|\geq t)\leq e^{-2}$ for all $i$. We may also assume that $t\geq t^*(k,X)$.
Integration by parts and Lemma \ref{lem:dil1} yield
\begin{align*}
\Ex|X_i|\ind_{\{|X_i| \ge t\}}&\leq 3t\Pr(|X_i|\geq t)+t\int_0^{\infty}\Pr(X_i \ge (s+3)t) ds 
\\
&\leq 3t\Pr(|X_i|\geq t)+t\int_0^{\infty}\Pr(|X_i|\ge t) e^{-s}ds\leq
4t \Pr(|X_i|\geq t).
\end{align*}
Therefore
\[
\sum_{i=1}^{n} \Ex|X_i|\ind_{\{|X_i| \ge t\}} \leq 4t\sum_{i=1}^{n} \Pr(|X_i| \ge t)
\leq 4t\sum_{i=1}^{n} \Pr(|X_i| \ge t^*(k,X))\leq 4kt,
\]
so $t(k,X)\leq 4C_{11}t^*(k,X)$.
\end{proof}

\end{document}